\documentclass[12pt,reqno]{article}

\usepackage[usenames]{color}
\usepackage{amssymb}
\usepackage{amsmath}
\usepackage{amsthm}
\usepackage{amsfonts}
\usepackage{amscd}
\usepackage{graphicx}

\usepackage[colorlinks=true,
linkcolor=webgreen,
filecolor=webbrown,
citecolor=webgreen]{hyperref}

\definecolor{webgreen}{rgb}{0,.5,0}
\definecolor{webbrown}{rgb}{.6,0,0}

\usepackage{color}
\usepackage{fullpage}
\usepackage{float}

\usepackage{graphics}
\usepackage{latexsym}

\begin{document}

\theoremstyle{plain}
\newtheorem{theorem}{Theorem}
\newtheorem{corollary}[theorem]{Corollary}
\newtheorem{proposition}{Proposition}
\newtheorem{lemma}{Lemma}
\newtheorem*{example}{Examples}
\newtheorem{remark}{Remark}

\numberwithin{equation}{section}

\begin{center}

\vskip 1cm{\LARGE\bf  
A Short Proof of Knuth's Old Sum 
}
\vskip 1cm
\large
Kunle Adegoke \\ 
Department of Physics and Engineering Physics\\
Obafemi Awolowo University\\
220005 Ile-Ife \\
Nigeria \\
\href{mailto:adegoke00@gmail.com }{\tt adegoke00@gmail.com}\\

\end{center}

\vskip .2 in

\begin{abstract}
We give a short proof of the well-known Knuth's old sum and provide some generalizations. Our approach utilizes the binomial theorem and integration formulas derived using the Beta function. Several new polynomial identities and combinatorial identities are derived.
\end{abstract}



\section{Introduction}
There appears to be a renewed interest~\cite{rathie22, tefera23, lim23, alzer24, adegoke24} in the famous Knuth's old sum (also known as the Reed Dawson identity),
\begin{equation}\label{eq.knuth}
\sum_{k = 0}^n {( - 1)^k \binom{{n}}{k}2^{ - k} \binom{{2k}}{k}}
 =  \begin{cases}
 2^{ - n} \binom{{n}}{n/2}, &\text{if $n$ is even};\\ 
 0,&\text{if $n$ is odd}. \\ 
 \end{cases} 
\end{equation}
Many different proofs of this identity and various generalizations exist in the literature (see~\cite{prodinger94} for a survey).

In this paper we give a very short proof of~\eqref{eq.knuth} and offer the following generalization:
\begin{equation}\label{eq.knuth-gen}
\begin{split}
&\sum_{k = 0}^n {( - 1)^k \binom{{n}}{k}2^{ - k - m} \binom{{2(k + m)}}{{k + m}}} \\ 
&\qquad= \begin{cases}
 \sum_{k = 0}^{\left\lfloor {m/2} \right\rfloor } {\binom{{m}}{{2k}}2^{ - n - 2k} \binom{{2k + n}}{{(2k + n)/2}}},&\text{if $n$ is even} ; \\ 
 -\sum_{k = 1}^{\left\lceil {m/2} \right\rceil } {\binom{{m}}{{2k - 1}}2^{ - n - 2k + 1} \binom{{2k + n - 1}}{{(2k + n - 1)/2}}},&\text{if $n$ is odd} ; \\ 
 \end{cases} 
\end{split}
\end{equation}
where $m$ and $n$ are non-negative integers and, as usual, $\lfloor z\rfloor$ is the greatest integer less than or equal to $z$ while $\lceil z\rceil$ is the smallest integer greater than or equal to $z$.

The following special cases of~\eqref{eq.knuth-gen} were also reported in Riordan~\cite[p.72, Problem 4(b)]{riordan71}:
\begin{gather}
\sum_{k = 0}^{\left\lfloor {n/2} \right\rfloor } {\binom n{2k}2^{n - 2k}\binom{2k}k}  =\binom{2n}n \label{eq.ilslov7},\\
\sum_{k = 1}^{\left\lceil {n/2} \right\rceil } {\binom{{n}}{{2k - 1}}2^{n - 2k} \binom{{2k}}{k}}  = \frac12 \binom{{2n + 2}}{{n + 1}} - \binom{{2n}}{n}=\frac n{n + 1}\binom{2n}n\label{eq.ef9et5k}.
\end{gather}
Identity~\eqref{eq.ilslov7} corresponds to setting $n=0$ in~\eqref{eq.knuth-gen} and re-labeling $m$ as $n$; while~\eqref{eq.ef9et5k} follows from setting $n=1$ in~\eqref{eq.knuth-gen}.

In section~\ref{sec.complement}, we will derive the following complements of Knuth's old sum:
\begin{equation*}
\sum_{k = 0}^n {( - 1)^k \binom{{2k}}{k}\binom{{2\left( {n - k} \right)}}{{n - k}}}
=  \begin{cases}
 2^n \binom{n}{n/2},&\text{if $n$ is even;} \\ 
 0,&\text{if $n$ is odd;} \\ 
 \end{cases} 
\end{equation*}
and 
\begin{equation}\label{eq.complement2}
\sum_{k=0}^n{\binom{2\left(n - k\right)}{n-k}\binom{2k}k}=2^{2n}.
\end{equation}
Identity~\eqref{eq.complement2} is the famous combinatorial identity concerning the convolution of central binomial coefficients. Many different proofs of this identity exist in the literature, (see Miki\'c~\cite{mikic} and the many references therein).

Identity~\eqref{eq.knuth-gen} is itself a particular case of a more general identity, stated in Theorem~\ref{thm.gcd6vin}, which has many interesting consequences, including another generalization of Knuth's old sum, namely,
\begin{equation*}
\sum_{k = 0}^n {( - 1)^k \binom{{n}}{k}2^{ - k} \binom{{2k + v}}{{\left( {2k + v} \right)/2}}\binom{{k + v}}{{v/2}}^{ - 1} }  
=  \begin{cases}
 2^{ - n} \binom{{n}}{{n/2}}\binom{{\left( {n + v} \right)/2}}{{v/2}}^{ - 1},&\text{if $n$ is even;}  \\ 
 0,&\text{if $n$ is odd;} \\ 
 \end{cases} 
\end{equation*}
where $v$ is a real number; as well as simple, apparently new combinatorial identities such as
\begin{equation*}
\sum_{k = 1}^{\left\lceil {n/2} \right\rceil } {\binom{n}{2k - 1}\,2^{n - 2k}\, C_k }  = \frac{1}{2}\,C_{n + 2}  - C_{n + 1};
\end{equation*}
where, here and throughout this paper, 
\begin{equation*}
C_j  = \frac 1{j + 1}\,\binom{2j}j,
\end{equation*}
defined for every non-negative integer $j$, is a Catalan number.

Based on the binomial theorem, we will derive, in Section~\ref{polynomials}, some presumably new polynomial identities, including the following:
\begin{equation}\label{eq.l5xib79}
\sum_{k = 0}^n {( - 1)^{n - k} \binom{{n}}{k}2^{ - k} \binom{{2k}}{k}\left( {1 - x} \right)^{n - k} }  = \sum_{k = 0}^{\left\lfloor {n/2} \right\rfloor } {\binom{{n}}{{2k}}2^{ - 2k} \binom{{2k}}{k}x^{n - 2k} }.
\end{equation}
Identity~\eqref{eq.l5xib79} subsumes Knuth's old sum~\eqref{eq.knuth} (at $x=0$), as well as~\eqref{eq.ilslov7} (at $x=1$).

Finally, in Section~\ref{combinatorial}, the polynomial identities will facilitate the derivation of apparently new combinatorial identities such as
\begin{gather*}
\sum_{k = 0}^{\left\lfloor {n/2} \right\rfloor } {\binom{{n}}{{2k}}\frac{1}{{2k + 1}}}  = \frac{{2^{n - 1} }}{{2^n  - 1}}\sum_{k = 1}^{\left\lceil {n/2} \right\rceil } {\binom{{n}}{{2k - 1}}\frac{1}{k}},\quad n\ne 0,\\
\sum_{k = 0}^{\left\lfloor {n/2} \right\rfloor } {\binom{{n}}{{2k}}2^{ - 2k} C_k }  = \frac{{2^{ - n + 1} }}{{n + 2}}\left( {2n + 1} \right)C_n ,
\end{gather*} 
and
\begin{equation*}
\sum_{k = 0}^n {( - 1)^{n - k} \binom{{n}}{k}\frac{{2\left( {2k + 1} \right)}}{{k + 2}}C_k }  = \sum_{k = 0}^{\left\lfloor {n/2} \right\rfloor } {\binom{{n}}{{2k}}C_k } .
\end{equation*}
\section{Required identities}
In order to give the short proof of Knuth's old sum, we need a couple of definite integrals which we establish in Lemma~\ref{lem.ibicayr}.

The binomial coefficients are defined, for non-negative integers $m$ and $n$, by
\begin{equation*}
\binom mn=
\begin{cases}
\dfrac{{m!}}{{n!(m - n)!}}, & \text{$m \ge n$};\\
0, & \text{$m<n$};
\end{cases}
\end{equation*}
the number of distinct sets of $n$ objects that can be chosen from $m$ distinct objects.

Generalized binomial coefficients are defined for complex numbers $u$ and $v$, excluding the set of negative integers, by
\begin{equation}\label{eq.o0ohqay}
\binom{{u}}{v} = \frac{{\Gamma \left( {u + 1} \right)}}{{\Gamma \left( {v + 1} \right)\Gamma \left( {u - v + 1} \right)}},
\end{equation}
where $\Gamma(z)$ is the Gamma function defined by
\begin{equation*}
\Gamma (z) = \int_0^\infty  {e^{ - t} t^{z - 1} dt}  = \int_0^\infty  {\left( {\log \left( {1/t} \right)} \right)^{z - 1} dt}
\end{equation*}
and extended to the rest of the complex plain, excluding the non-positive integers, by analytic continuation.

\begin{lemma}\label{lem.ibicayr}
Let $u$ and $v$ be complex numbers such that $\Re u>-1$ and $\Re v> -1$. Let $m$ be a non-negative integer. Then
\begin{equation}
\int_0^\pi  {\cos ^u (x/2)}\,dx =2^{-u}\,\pi\binom{u}{u/2}=\int_0^\pi  {\sin ^u (x/2)}\,dx\label{eq.cejya7g},
\end{equation}
\begin{equation}\label{eq.s4qk1l3}
\int_0^\pi  {\cos ^m x\,dx} 
 = \begin{cases}
2^{-m}\,\pi\binom{{m}}{{m/2}},&\text{if $m$ is even;} \\
0,&\text{if $m$ is odd;} \\ 
 \end{cases} 
\end{equation}
and, more generally,
\begin{equation}\label{int1}
I(u,v):=\int_0^\pi  {\cos ^u \left( {\frac{x}{2}} \right)\sin ^v \left( {\frac{x}{2}} \right)\,dx}  = 2^{-u -v}\,\pi\,\binom{{u}}{{u/2}}\binom{{v}}{{v/2}}\binom{{\left( {u + v} \right)/2}}{{u/2}}^{ - 1}, 
\end{equation}
and
\begin{equation}\label{int2}
J(m,v):=\int_0^\pi  {\cos ^m x\sin ^v x\,dx}  =  
\begin{cases}
 2^{-m -v}\,\pi\,\binom{{m}}{{m/2}}\binom{{v}}{{v/2}}\binom{{\left( {m + v} \right)/2}}{{m/2}}^{ - 1}, &\text{if $m$ is even;}\\ 
 0, &\text{if $m$ is odd.}  \\ 
\end{cases} 
\end{equation}
Obviously $I(v,u)=I(u,v)$, a symmetry property that is not possessed by $J(m,v)$.

\end{lemma}
\begin{proof}
Identities~\eqref{int1} and~\eqref{int2} are immediate consequences of the well-known Beta function integral~\cite[Entry 3.621.5, Page 397]{gradsh07}:
\begin{equation}\label{beta}
K(u,v):=\int_0^{\pi/2}  {\cos ^u x\sin ^v x\,dx}  = 2^{ - u - v - 1}\, \pi \binom{{u}}{{u/2}}\binom{{v}}{{v/2}}\binom{{\left( {u + v} \right)/2}}{{u/2}}^{ - 1} , 
\end{equation}
valid for $\Re u>-1$, $\Re v> -1$, with the symmetry property $K(u,v)=K(v,u)$.

Identity~\eqref{int1} is obtained via a simple change of the integration variable from $x$ to $y$ in~\eqref{beta}, with $x=y/2$. 

To prove~\eqref{int2}, write
\begin{equation*}
J (m,v)= \int_0^\pi  {\cos ^m x\sin ^v x\,dx}  = \int_0^{\pi /2} {\cos ^m x\sin ^v x\,dx}  + \int_{\pi /2}^\pi  {\cos ^m x\sin ^v x\,dx}.
\end{equation*}
Change the integration variable in the second integral on the right hand side from $x$ to $y$ via $x=y+\pi/2$; this gives
\begin{equation*}
\begin{split}
J (m,v)&= \int_0^{\pi /2} {\cos ^m x\sin ^v x\,dx}  + ( - 1)^m \int_0^{\pi /2} {\sin ^m y\cos ^v y\,dy}\\ 
 &= K(m,v) + ( - 1)^m K(v,m)\\
 &= \left( {1 + ( - 1)^m } \right)K(m,v);
\end{split}
\end{equation*}
and hence~\eqref{int2}.

\end{proof}

\begin{remark}
Since, for a real number $u$,
\begin{equation*}
1 + (-1)^u=2\cos^2\left(\frac{\pi u}2\right) + i\sin\left(\pi u\right),
\end{equation*}
the $J(m,v)$ stated in~\eqref{int2} is a special case of the following more general result:
\begin{equation}
\begin{split}
J\left( {u,v} \right) &= \int_0^\pi  {\cos ^u x\sin ^v x\,dx}\\
&  = \frac{\pi }{{2^{u + v + 1} }}\binom{{u}}{{u/2}}\binom{{v}}{{v/2}}\binom{{\left( {u + v} \right)/2}}{{u/2}}^{ - 1} \left( {2\cos ^2 \left( {\frac{{\pi u}}{2}} \right) + i\sin \left( {\pi u} \right)} \right),
\end{split}
\end{equation}
which is valid for $u>-1$ and $\Re v>-1$.
\end{remark}

\section{A short proof of Knuth's old sum}
\begin{theorem}
If $n$ is a non-negative integer, then
\begin{equation*}
\sum_{k = 0}^n {( - 1)^k \binom{{n}}{k}2^{ - k} \binom{{2k}}{k}}
 =  \begin{cases}
 2^{ - n} \binom{{n}}{n/2}, &\text{if $n$ is even};\\ 
 0,&\text{if $n$ is odd}. \\ 
 \end{cases} 
\end{equation*}
\end{theorem}
\begin{proof}
Substitute $-\cos x - 1$ for $y$ in the binomial theorem
\begin{equation*}
\sum_{k=0}^n{\binom nk y^k}=(1 + y)^n,
\end{equation*}
to obtain
\begin{equation}\label{eq.ff6aeas}
\sum_{k = 0}^n {( - 1)^k \binom nk2^k \cos ^{2k} (x/2)}  = ( - 1)^n \cos ^n x.
\end{equation}
Thus
\begin{equation*}
\sum_{k = 0}^n {( - 1)^k \binom nk2^k \int_0^\pi  {\cos ^{2k} (x/2)\,dx} }  = ( - 1)^n \int_0^\pi  {\cos ^n x\,dx},
\end{equation*}
and hence~\eqref{eq.knuth} on account of~\eqref{eq.cejya7g} and~\eqref{eq.s4qk1l3}.
\end{proof}

\section{A generalization of Knuth's old sum}
In this section we extend~\eqref{eq.knuth} by introducing an arbitrary non-negative integer $m$ and a real number $v$.
\begin{theorem}\label{thm.gcd6vin}
If $m$ and $n$ are non-negative integers and $v$ is a real number, then
\begin{equation}\label{eq.r9e10tq}
\begin{split}
&\sum_{k = 0}^n {( - 1)^k \binom{{n}}{k}2^{ - k - m} \binom{{2k + 2m + v}}{{(2k + 2m + v)/2}}}\binom{k + m + v}{v/2}^{-1} \\ 
&\qquad= \begin{cases}
 \sum_{k = 0}^{\left\lfloor {m/2} \right\rfloor } {\binom{{m}}{{2k}}2^{ - n - 2k} \binom{{2k + n}}{{(2k + n)/2}}}\binom{{(2k + n + v)/2}}{{(2k + n)/2}}^{-1},&\text{if $n$ is even} ; \\ 
 -\sum_{k = 1}^{\left\lceil {m/2} \right\rceil } {\binom{{m}}{{2k - 1}}2^{ - n - 2k + 1} \binom{{2k + n - 1}}{{(2k + n - 1)/2}}}\binom{{(2k + n - 1 + v)/2}}{{(2k + n - 1)/2}}^{-1},&\text{if $n$ is odd} . \\ 
 \end{cases} 
\end{split}
\end{equation}
\end{theorem}
In particular,
\begin{equation*}
\begin{split}
&\sum_{k = 0}^n {( - 1)^k \binom{{n}}{k}2^{ - k - m} \binom{{2(k + m)}}{{k + m}}} \\ 
&\qquad= \begin{cases}
 \sum_{k = 0}^{\left\lfloor {m/2} \right\rfloor } {\binom{{m}}{{2k}}2^{ - n - 2k} \binom{{2k + n}}{{(2k + n)/2}}},&\text{if $n$ is even} ; \\ 
 -\sum_{k = 1}^{\left\lceil {m/2} \right\rceil } {\binom{{m}}{{2k - 1}}2^{ - n - 2k + 1} \binom{{2k + n - 1}}{{(2k + n - 1)/2}}},&\text{if $n$ is odd} . \\ 
 \end{cases} 
\end{split}
\end{equation*}
\begin{proof}
Since
\begin{equation*}
\left( {1 + \cos x} \right)^m  = 2^m \cos ^{2m} \left( {\frac{x}{2}} \right) = \sum_{k = 0}^m {\binom{m}{k}\cos ^k x}
\end{equation*}
and
\begin{equation*}
\sin ^v x = 2^v \sin ^v \left( {\frac{x}{2}} \right)\cos ^v \left( {\frac{x}{2}} \right),
\end{equation*}
multiplication of the left hand side of~\eqref{eq.ff6aeas} by
\begin{equation*}
2^{m + v} \cos ^{2m + v} \left( {\frac{x}{2}} \right)\sin ^v \left( {\frac{x}{2}} \right)
\end{equation*}
and the right hand side by
\begin{equation*}
\sin ^v x\sum_{k = 0}^m {\binom{m}{k}\cos ^k x}
\end{equation*}
gives
\begin{equation*}
\begin{split}
&\sum_{k = 0}^n {( - 1)^k \binom{{n}}{k}2^{k + m + v} \cos ^{2k + 2m + v} (x/2)\sin ^v (x/2)} \\
&\qquad = ( - 1)^n \sum_{k = 0}^m {\binom{{m}}{k}\cos ^{k + n} x\sin^v x} ;
\end{split}
\end{equation*}
so that
\begin{equation*}
\begin{split}
&\sum_{k = 0}^n {( - 1)^k \binom{{n}}{k}2^{k + m + v} \cos ^{2k + 2m + v} (x/2)\sin ^v (x/2)} \\
&\qquad = ( - 1)^n \sum_{k = 0}^{\left\lfloor {m/2} \right\rfloor } {\binom{{m}}{{2k}}\cos ^{2k + n} x\sin ^v x}  + ( - 1)^n \sum_{k = 1}^{\left\lceil {m/2} \right\rceil } {\binom{{m}}{{2k - 1}}\cos ^{2k - 1 + n} x\sin ^v x} ,
\end{split}
\end{equation*}
from which~\eqref{eq.r9e10tq} now follows by termwise integration from $0$ to $\pi$, according to the parity of~$n$, using~Lemma~\ref{lem.ibicayr}.

\end{proof}

\begin{corollary}
If $n$ is a non-negative integer and $v$ is a real number, then
\begin{equation}\label{eq.hdj69wz}
\sum_{k = 0}^n {( - 1)^k \binom{{n}}{k}2^{ - k} \binom{{2k + v}}{{\left( {2k + v} \right)/2}}\binom{{k + v}}{{v/2}}^{ - 1} }  
=  \begin{cases}
 2^{ - n} \binom{{n}}{{n/2}}\binom{{\left( {n + v} \right)/2}}{{v/2}}^{ - 1},&\text{if $n$ is even;}  \\ 
 0,&\text{if $n$ is odd;} \\ 
 \end{cases} 
\end{equation}

\end{corollary}

\begin{corollary}
If $n$ is a non-negative integer and $v$ is a real number, then
\begin{equation}\label{eq.y6pnymc}
\sum_{k = 0}^{\left\lfloor {n/2} \right\rfloor } {\binom{{n}}{{2k}}2^{ - 2k} \binom{{2k}}{k}\binom{{\left( {2k + v} \right)/2}}{k}^{ - 1} }  = 2^{ - n} \binom{{2n + v}}{{\left( {2n + v} \right)/2}}\binom{{n + v}}{{v/2}}^{ - 1},
\end{equation}
and
\begin{equation}\label{eq.wf0mlz9}
\begin{split}
&\sum_{k = 1}^{\left\lceil {n/2} \right\rceil } {\binom{{n}}{{2k - 1}}2^{n - 2k} \binom{{2k}}{k}\binom{{\left( {2k + v} \right)/2}}{k}^{ - 1} }\\
&\qquad  =\frac12\, \binom{{2n + v + 2}}{{\left( {2n + v + 2} \right)/2}}\binom{{n + v + 1}}{{v/2}}^{ - 1}  - \binom{{2n + v}}{{\left( {2n + v} \right)/2}}\binom{{n + v}}{{v/2}}^{-1}.
\end{split}
\end{equation}
\end{corollary}

\begin{proof}
Identity~\eqref{eq.y6pnymc} is obtained by setting $n=0$ in~\eqref{eq.r9e10tq} and re-labeling $m$ as $n$ while~\eqref{eq.wf0mlz9} is the evaluation of~\eqref{eq.r9e10tq} at $n=1$ with a re-labeling of $m$ as $n$.
\end{proof}

\begin{proposition}
If $n$ is a non-negative integer, then
\begin{gather}
\sum_{k = 1}^{\left\lceil {n/2} \right\rceil } {\binom{{n}}{{2k - 1}}\frac{1}{{2k + 1}}}  = \frac{{2^{n + 1} }}{{n + 2}} - \frac{{2^n }}{{n + 1}},\label{eq.amk3put}\\
\sum_{k = 1}^{\left\lceil {n/2} \right\rceil } {\binom{n}{2k - 1}\,2^{n - 2k}\, C_k }  = \frac{1}{2}\,C_{n + 2}  - C_{n + 1}\label{eq.ct31is7} .
\end{gather}

\end{proposition}

\begin{proof}
Evaluation of~\eqref{eq.wf0mlz9} at $v=1$ gives~\eqref{eq.amk3put} while evaluation at $v=2$ yields~\eqref{eq.ct31is7}. In deriving~\eqref{eq.amk3put}, we used the following relationships between binomial coefficients:
\begin{gather}
\binom{{r}}{{1/2}} = \frac{{2^{2r + 1} }}{\pi }\,\binom{{2r}}{r}^{ - 1} ,\label{eq.muz1im8}\\
\binom{{r}}{{r/2}} = \frac{{2^{2r} }}{\pi }\,\binom{{r}}{{\left( {r - 1} \right)/2}}^{-1},\\
\binom{{r + 1/2}}{r} = \left( {2r + 1} \right)\,2^{ - 2r} \,\binom{{2r}}{r} ,
\end{gather}
and
\begin{equation}\label{eq.mv0q30s}
r\binom sr=s\binom {s - 1}{r - 1};
\end{equation}
all of which can be derived by using the Gamma function identities:
\begin{equation*}
\Gamma \left( {u + \frac{1}{2}} \right) = \sqrt \pi\, 2^{-2u}\binom{{2u}}{u}\,\Gamma \left( {u + 1} \right),
\end{equation*}
and
\begin{equation*}
\Gamma \left( { - u + \frac{1}{2}} \right) = ( - 1)^u\, 2^{2u} \,\binom{{2u}}{u}^{ - 1} \frac{{\sqrt \pi  }}{{\Gamma \left( {u + 1} \right)}},
\end{equation*}
together with the definition of the generalized binomial coefficients as given in~\eqref{eq.o0ohqay}.
\end{proof}

\begin{proposition}
If $m$ and $n$ are non-negative integers, then
\begin{equation}
\sum_{k = 0}^n {( - 1)^k \binom{{n}}{k}\frac{{2^{k + m} }}{{k + m + 1}}}  =  \begin{cases}
 \sum_{k = 0}^{\left\lfloor {m/2} \right\rfloor } {\binom{{m}}{{2k}}\frac{{1}}{{2k + n + 1}}},&\text{if $n$ is even};  \\ 
 -\sum_{k = 1}^{\left\lceil {m/2} \right\rceil } {\binom{{m}}{{2k - 1}}\frac{{1}}{{2k + n}}},&\text{if $n$ is odd}.  \\ 
 \end{cases} 
\end{equation}
\end{proposition}
In particular,
\begin{equation}
\sum_{k = 0}^n {\frac{{( - 1)^k \binom{{n}}{k}2^k }}{{k + 1}}}  =  \begin{cases}
 \frac1{n + 1},&\text{if $n$ is even};  \\ 
 0,&\text{if $n$ is odd};  \\ 
 \end{cases} 
\end{equation}
and
\begin{equation}
\sum_{k = 0}^n {\frac{{( - 1)^k \binom{{n}}{k}2^{k + 1}}}{{k + 2}}}  =  \begin{cases}
 \frac1{n + 1},&\text{if $n$ is even};  \\ 
 -\frac1{n + 2},&\text{if $n$ is odd}.  \\ 
 \end{cases} 
\end{equation}

\begin{proof}
Evaluate~\eqref{eq.r9e10tq} at $v=1$.

\end{proof}

\section{Complements of Knuth's old sum}\label{sec.complement}

\begin{theorem}
If $n$ is a non-negative integer, then
\begin{equation}\label{eq.complement1}
\sum_{k = 0}^n {( - 1)^k \binom{{2k}}{k}\binom{{2\left( {n - k} \right)}}{{n - k}}}
=  \begin{cases}
 2^n \binom{n}{n/2},&\text{if $n$ is even;} \\ 
 0,&\text{if $n$ is odd.} \\ 
 \end{cases} 
\end{equation}

\end{theorem}

\begin{proof}
Set $a=\cos^2(x/2)$ and $b=-\sin^2(x/2)$ in the binomial theorem
\begin{equation}\label{eq.jch33hz}
\sum_{k = 0}^n {\binom{{n}}{k}a^k b^{n - k} }  = \left( {a + b} \right)^n,
\end{equation}
to obtain
\begin{equation}\label{ej40f4q}
\sum_{k = 0}^n {( - 1)^{n - k} \binom{{n}}{k}\cos ^{2k} \left( {\frac{x}{2}} \right)\sin ^{2n - 2k} \left( {\frac{x}{2}} \right)}  = \cos ^n x,
\end{equation}
from which~\eqref{eq.complement1} follows by term-wise integration using Lemma~\ref{lem.ibicayr}.
\end{proof}

\begin{theorem}
If $n$ is a non-negative integer, then
\begin{equation*}
\sum_{k=0}^n{\binom{2n - 2k}{n-k}\binom{2k}k}=2^{2n}.
\end{equation*}

\end{theorem}

\begin{proof}
Set $a=\cos^2(x/2)$ and $b=\sin^2(x/2)$ in the binomial theorem~\eqref{eq.jch33hz} to obtain
\begin{equation}\label{eq.zkzhma8}
\sum_{k = 0}^n {\binom{{n}}{k}\cos ^{2k} \left( {\frac{x}{2}} \right)\sin ^{2n - 2k} \left( {\frac{x}{2}} \right)}  = 1,
\end{equation}
from which the stated identity follows by term-wise integration using Lemma~\ref{lem.ibicayr}.
\end{proof}

Next, we present a generalization of~\eqref{eq.complement1}.
\begin{theorem}
If $n$ is a non-negative integer and $v$ is a real number, then
\begin{equation}
\begin{split}
\sum_{k = 0}^n {( - 1)^k \binom{{n}}{k}\binom{{2k + v}}{{\left( {2k + v} \right)/2}}\binom{{2n - 2k + v}}{{\left( {2n - 2k + v} \right)/2}}\binom{{n + v}}{{\left( {2k + v} \right)/2}}^{ - 1} } \\
 =  
\begin{cases}
 2^n \binom{{n}}{{n/2}}\binom{{v}}{{v/2}}\binom{{\left( {n + v} \right)/2}}{{v/2}}^{-1},&\text{if $n$ is even;} \\ 
 0,&\text{if $n$ is odd.} \\ 
\end{cases} 
 .
\end{split}
\end{equation}

\end{theorem}

\begin{proof}
Multiply through~\eqref{ej40f4q} by $\sin^v x$ and integrate from $0$ to $\pi$, using Lemma~\ref{lem.ibicayr}.
\end{proof}
We conclude this section with a generalization of~\eqref{eq.complement2}.
\begin{theorem}
If $n$ is a non-negative integer and $v$ is a real number, then
\begin{equation}
\sum_{k = 0}^n {\binom{{n}}{k}\binom{{2k + v}}{{\left( {2k + v} \right)/2}}\binom{{2n - 2k + v}}{{\left( {2n - 2k + v} \right)/2}}\binom{{n + v}}{{\left( {2k + v} \right)/2}}^{ - 1} } =2^{2n}\binom v{v/2}.
\end{equation}

\end{theorem}

\begin{proof}
Multiply through~\eqref{eq.zkzhma8} by $\sin^v x$ and integrate from $0$ to $\pi$, using Lemma~\ref{lem.ibicayr}.
\end{proof}

\section{Combinatorial identities associated with polynomial identities of a certain type}\label{sec.combinat}
In this section we derive the combinatorial identities associated with any polynomial identity having the following form:
\begin{equation}\label{eq.zig6lng}
P(t,\ldots)=\sum_{k = s}^n {f(k)\left( {1 + t} \right)^{p(k)} }  = \sum_{k = m}^r {g(k)\,t^{q(k)} };
\end{equation}
where $m$, $n$, $r$ and $s$ are non-negative integers, $p(k)$ and $q(k)$ are sequences of non-negative integers, $f(k)$ and $g(k)$ are sequences, and $t$ is a complex variable. The ellipsis (\ldots) indicates the presence of other parameters and variables.

\begin{theorem}\label{thm.a2bugv6}
Let $P(t,\ldots)$ be the polynomial identity given in~\eqref{eq.zig6lng}. Let $u$ and $v$ be arbitrary complex numbers such that $\Re u>-1$ and $\Re v>-1$. Then
\begin{equation}\label{eq.rxkjh8x}
\sum_{k = s}^n {f(k)\binom{{p(k) + u + v + 1}}{{u + 1}}^{ - 1} }  = \frac{u + 1}{v + 1}\,\sum_{k = m}^r {( - 1)^{q(k)} g(k)\binom{{q(k) + u + v + 1}}{{v + 1}}^{-1}} .
\end{equation}

\end{theorem}
In particular,
\begin{equation}
\sum_{k = s}^n {\frac{{f(k)}}{{p(k) + 1}}}  = \sum_{k = m}^r {\frac{{( - 1)^{q(k)} g(k)}}{{q(k) + 1}}}.
\end{equation}
\begin{proof}
Write $-t$ for $t$ in~\eqref{eq.zig6lng} and multiply through by $t^u(1 - t)^v$ to obtain
\begin{equation*}
\sum_{k = s}^n {f(k)\left( {1 - t} \right)^{p(k) + v} t^u }  = \sum_{k = m}^r {( - 1)^{q(k)} g(k)\left( {1 - t} \right)^v t^{q(k) + u} };
\end{equation*}
from which~\eqref{eq.rxkjh8x} follows after integrating from $0$ to $1$, using the Beta function (variant of~\eqref{beta}):
\begin{equation}\label{eq.vmrnwwf}
\int_0^1 {\left( {1 - t} \right)^x t^y\, dt}  = \frac{1}{{x + 1}}\,\binom{{x + y + 1}}{{x + 1}}^{ - 1} ;
\end{equation}
for $\Re x>-1$ and $\Re y>-1$.
\end{proof}

\begin{theorem}
Let $P(t,\ldots)$ be the polynomial identity given in~\eqref{eq.zig6lng}. Let $u$ and $v$ be arbitrary complex numbers such that $\Re v>-1$, $\Re(2(u - p(j)) + v)>-1$, $\Re(2(u - q(j)) + v)>-1$ and $2q(j) + \Re v>-1$ for every non-negative integer $j$. Then
\begin{equation}\label{eq.a4q7ien}
\begin{split}
&\sum_{k = s}^n {f(k)2^{2p(k)} \binom{{2\left( {u - p(k)} \right) + v}}{{\left( {2\left( {u - p(k)} \right) + v} \right)/2}}\binom{{u - p(k) + v}}{{v/2}}^{ - 1} }\\
&\qquad  = \binom{{v}}{{v/2}}^{ - 1} \sum_{k = m}^r {g(k)\binom{{2\left( {u - q(k)} \right) + v}}{{\left( {2\left( {u - q(k)} \right) + v} \right)/2}}\binom{{2q(k) + v}}{{\left( {2q(k) + v} \right)/2}}\binom{{u + v}}{{\left( {2q(k) + v} \right)/2}}^{ - 1} } ,
\end{split}
\end{equation}
and
\begin{equation}\label{eq.ozzuug9}
\begin{split}
&\sum_{k = m}^r {( - 1)^{q(k)} g(k)2^{2q(k)} \binom{{2\left( {u - q(k)} \right) + v}}{{\left( {2\left( {u - q(k)} \right) + v} \right)/2}}\binom{{u - q(k) + v}}{{v/2}}^{ - 1} }\\
&\qquad  = \binom{{v}}{{v/2}}^{ - 1} \sum_{k = s}^n {( - 1)^{p(k)} f(k)\binom{{2\left( {u - p(k)} \right) + v}}{{\left( {2\left( {u - p(k)} \right) + v} \right)/2}}\binom{{2p(k) + v}}{{\left( {2p(k) + v} \right)/2}}\binom{{u + v}}{{\left( {2p(k) + v} \right)/2}}^{ - 1} } .
\end{split}
\end{equation}

\end{theorem}
In particular,
\begin{equation}
\sum_{k = s}^n {f(k)2^{2p(k)} \binom{{2\left( {u - p(k)} \right)}}{{u - p(k)}}}  = \sum_{k = m}^r {g(k)\binom{{2\left( {u - q(k)} \right)}}{{u - q(k)}}\binom{{2q(k)}}{{q(k)}}\binom{{u}}{{q(k)}}^{ - 1} } 
\end{equation}
and
\begin{equation}
\sum_{k = m}^r {( - 1)^{q(k)} g(k)2^{2q(k)} \binom{{2\left( {u - q(k)} \right)}}{{u - q(k)}}}  = \sum_{k = s}^n {( - 1)^{p(k)} f(k)\binom{{2\left( {u - p(k)} \right)}}{{u - p(k)}}\binom{{2p(k)}}{{p(k)}}\binom{{u}}{{p(k)}}^{ - 1} } .
\end{equation}

\begin{proof}
Substituting $t=y/x$ in~\eqref{eq.zig6lng} and multiplying through by $x^w$ gives
\begin{equation}\label{eq.ib3ix27}
\sum_{k = s}^n {f(k)x^{u - p(k)} \left( {x + y} \right)^{p(k)} }  = \sum_{k = m}^r {g(k)x^{u - q(k)} y^{q(k)} }.
\end{equation}
Writing $\cos^2x$ for $x$ and $\sin^2x$ for $y$ in~\eqref{eq.ib3ix27}, multiplying through by $\sin^vx$ and integrating from $0$ to $\pi/2$ using~Lemma~\ref{lem.ibicayr} gives~\eqref{eq.a4q7ien}. Identity~\eqref{eq.ozzuug9} follows from the fact that the transformation $y\to y-x$ followed by $x\to -x$ causes~\eqref{eq.zig6lng} to become
\begin{equation*}
\sum_{k = m}^r {( - 1)^{u - q(k)} g(k)x^{u - q(k)} \left( {x + y} \right)^{q(k)} }  = \sum_{k = s}^n {( - 1)^{u - p(k)} f(k)x^{u - p(k)} y^{p(k)} }.
\end{equation*}

\end{proof}

\begin{theorem}
Let $P(t,\ldots)$ be the polynomial identity given in~\eqref{eq.zig6lng}. Let $u$ and $v$ be arbitrary complex numbers such that $\Re v>-1$, $\Re{u} - p(j)>-1$, $p(j)+\Re(v)>-1$ and $\Re u - q(j)>-1$ for every non-negative integer $j$. Then
\begin{equation}\label{eq.kmlq5k6}
\sum_{k = s}^n {( - 1)^{p(k)} f(k)\binom{{u + v}}{{u - p(k)}}^{ - 1} }  = \frac{{u + v + 1}}{{v + 1}}\sum_{k = m}^r {( - 1)^{q(k)} g(k)\binom{{u - q(k) + 1}}{{v + 1}}^{ - 1} } .
\end{equation}

\end{theorem}
In particular,
\begin{equation}
\sum_{k = s}^n {( - 1)^{p(k)} f(k)\binom{{u}}{{p(k)}}^{ - 1} }  = \left( {u + 1} \right)\sum_{k = m}^r {( - 1)^{q(k)} \frac{{g(k)}}{{u - q(k) + 1}}} .
\end{equation}

\begin{proof}
Set $y=-1$ in~\eqref{eq.ib3ix27} and multiply through by $(1 - x)^v$ to obtain
\begin{equation*}
\sum_{k = s}^n {( - 1)^{p(k)} f(k)x^{u - p(k)} \left( {1 - x} \right)^{p(k) + v} }  = \sum_{k = m}^r {( - 1)^{q(k)} g(k)x^{u - q(k)} \left( {1 - x} \right)^v },
\end{equation*}
which upon integration from $0$ to $1$, using~\eqref{eq.vmrnwwf}, gives~\eqref{eq.kmlq5k6}.
\end{proof}

\begin{theorem}
Let $P(t,\ldots)$ be the polynomial identity given in~\eqref{eq.zig6lng}. Let $u$ and $v$ be arbitrary complex numbers such that $\Re u>-1$ and $\Re v>-1$. Then
\begin{equation}\label{eq.cnwt5zb}
\begin{split}
&\sum_{k = s}^n {\frac{{f(k)}}{{2^{2p(k)} }}\binom{{v}}{{v/2}}\binom{{2p(k) + u}}{{\left( {2p(k) + u} \right)/2}}\binom{{\left( {2p(k) + u + v} \right)/2}}{{v/2}}^{-1}}\\ &\qquad  = \sum_{k = m}^r {\frac{{( - 1)^{q(k)} g(k)}}{{2^{2q(k)} }}\binom{{u}}{{u/2}}\binom{{2q(k) + v}}{{\left( {2q(k) + v} \right)/2}}\binom{{\left( {2q(k) + u + v} \right)/2}}{{u/2}}^{-1}} .
\end{split}
\end{equation}

\end{theorem}
In particular,
\begin{equation}
\sum_{k = s}^n {\frac{{f(k)}}{{2^{2p(k)} }}\binom{{2p(k)}}{{p(k)}}}  = \sum_{k = m}^r {\frac{{( - 1)^{q(k)} g(k)}}{{2^{2q(k)} }}\binom{{2q(k)}}{{q(k)}}} .
\end{equation}

\begin{proof}
Write $-\sin^2 t$ for $t$ in~\eqref{eq.zig6lng} and multiply through by $\cos^ut\,\sin^vt$ to obtain
\begin{equation*}
\sum_{k = s}^n {f(k)\cos ^{2p(k) + u} t\sin ^v t}  = \sum_{k = m}^r {( - 1)^{q(k)} g(k)\cos ^u t\sin ^{2q(k) + v} t},
\end{equation*}
from which~\eqref{eq.cnwt5zb} follows upon integration from $0$ to $\pi/2$ using Lemma~\ref{lem.ibicayr}.
\end{proof}

\begin{theorem}\label{thm.bgu7pnr}
Let $P(t,\ldots)$ be the polynomial identity given in~\eqref{eq.zig6lng}. Let $v$ be an arbitrary complex number such that $\Re v>-1$.

\begin{enumerate}

\item Suppose that, for every integer $j$, each of $q(2j)$ and $q(2j-1)$ is a sequence of non-negative integers having a definite parity but such that the parity of $q(2j)$ is different from the parity of $q(2j-1)$ for every integer $j$. 

If $q(2j)$ is an even integer for every integer $j$, then
\begin{equation}\label{eq.yv4kyfa}
\begin{split}
&\sum_{k = s}^n {\frac{{f(k)}}{{2^{p(k)} }}\binom{{2p(k) + v}}{{\left( {2p(k) + v} \right)/2}}\binom{{p(k) + v}}{{v/2}}^{ - 1} }\\  &\qquad= \sum_{k = \left\lfloor {(m + 1)/2} \right\rfloor }^{\left\lfloor {r/2} \right\rfloor } {\frac{{g(2k)}}{{2^{q(2k)} }}\binom{{q(2k)}}{{q(2k)/2}}\binom{{\left( {q(2k) + v} \right)/2}}{{v/2}}^{ - 1} },
\end{split}
\end{equation}
while if $q(2j)$ is an odd integer for every integer $j$, then
\begin{equation}\label{eq.ba9x4o9}
\begin{split}
&\sum_{k = s}^n {\frac{{f(k)}}{{2^{p(k)} }}\binom{{2p(k) + v}}{{\left( {2p(k) + v} \right)/2}}\binom{{p(k) + v}}{{v/2}}^{ - 1} }\\&\qquad  = \sum_{k = \left\lfloor {(m + 2)/2} \right\rfloor }^{\left\lceil {r/2} \right\rceil } {\frac{{g(2k - 1)}}{{2^{q(2k - 1)} }}\binom{{q(2k - 1)}}{{q(2k - 1)/2}}\binom{{\left( {q(2k - 1) + v} \right)/2}}{{v/2}}^{ - 1} } .
\end{split}
\end{equation}

\item Suppose that, for every integer $j$, each of $p(2j)$ and $p(2j-1)$ is a sequence of non-negative integers having a definite parity but such that the parity of $p(2j)$ is different from the parity of $p(2j-1)$ for every integer $j$. 

If $p(2j)$ is an even integer for every integer $j$, then
\begin{equation}\label{eq.kx7oxs0}
\begin{split}
&\sum_{k = m}^r {\frac{{g(k)( - 1)^{q(k)} }}{{2^{q(k)} }}\binom{{2q(k) + v}}{{\left( {2q(k) + v} \right)/2}}\binom{{q(k) + v}}{{v/2}}^{ - 1} }\\&\qquad  = \sum_{k = \left\lfloor {(s + 1)/2} \right\rfloor }^{\left\lfloor {n/2} \right\rfloor } {\frac{{( - 1)^{p(2k)} f(2k)}}{{2^{p(2k)} }}\binom{{p(2k)}}{{p(2k)/2}}\binom{{\left( {p(2k) + v} \right)/2}}{{v/2}}^{ - 1} } 
\end{split}
\end{equation}
while if $p(2j)$ is an odd integer for every integer $j$, then
\begin{equation}\label{eq.wnufxrh}
\begin{split}
&\sum_{k = m}^r {\frac{{g(k)( - 1)^{q(k)} }}{{2^{q(k)} }}\binom{{2q(k) + v}}{{\left( {2q(k) + v} \right)/2}}\binom{{q(k) + v}}{{v/2}}^{ - 1} }  \\&\qquad= \sum_{k = \left\lfloor {(s + 2)/2} \right\rfloor }^{\left\lceil {n/2} \right\rceil } {\frac{{( - 1)^{p(2k - 1)} f(2k - 1)}}{{2^{p(2k - 1)} }}\binom{{p(2k - 1)}}{{p(2k - 1)/2}}\binom{{\left( {p(2k - 1) + v} \right)/2}}{{v/2}}^{ - 1} } .
\end{split}
\end{equation}

\end{enumerate}

\end{theorem}

\begin{proof}
Set $t=\cos x$ in~\eqref{eq.zig6lng} and multiply through by $\sin^vx$ to obtain
\begin{equation*}
\begin{split}
&\sum_{k = s}^n {2^{p(k) + v} f(k)\cos ^{2p(k) + v} \left( {\frac{x}{2}} \right)\sin ^v \left( {\frac{x}{2}} \right)} \\
&\qquad= \sum_{k = \left\lfloor {(m + 1)/2} \right\rfloor }^{\left\lfloor {r/2} \right\rfloor } {g(2k)\cos ^{q(2k)} x\sin ^v x}  + \sum_{k = \left\lfloor {(m + 2)/2} \right\rfloor }^{\left\lceil {r/2} \right\rceil } {g(2k - 1)\cos ^{q(2k - 1)} x\sin ^v x},
\end{split}
\end{equation*}
from which~\eqref{eq.yv4kyfa} and~\eqref{eq.ba9x4o9} follow after term-wise integration from $0$ to $\pi$, using Lemma~\ref{lem.ibicayr}. Identities~\eqref{eq.kx7oxs0} and~\eqref{eq.wnufxrh} are obtained from~\eqref{eq.yv4kyfa} and~\eqref{eq.ba9x4o9} since~\eqref{eq.zig6lng} can be written in the following equivalent form:
\begin{equation*}
\sum_{k = m}^r {( - 1)^{q(k)} g(k)\left( {1 + t} \right)^{q(k)} }  = \sum_{k = s}^n {( - 1)^{p(k)} f(k)t^{p(k)} } .
\end{equation*}

\end{proof}

In particular,

\begin{enumerate}

\item Suppose that, for every integer $j$, each of $q(2j)$ and $q(2j-1)$ is a sequence of non-negative integers having a definite parity but such that the parity of $q(2j)$ is different from the parity of $q(2j-1)$ for every integer $j$. 

If $q(2j)$ is an even integer for every integer $j$, then
\begin{equation}
\sum_{k = s}^n {\frac{{f(k)}}{{2^{p(k)} }}\binom{{2p(k)}}{{p(k)}}}  = \sum_{k = \left\lfloor {(m + 1)/2} \right\rfloor }^{\left\lfloor {r/2} \right\rfloor } {\frac{{g(2k)}}{{2^{q(2k)} }}\binom{{q(2k)}}{{q(2k)/2}}} ,
\end{equation}
while if $q(2j)$ is an odd integer for every integer $j$, then
\begin{equation}
\sum_{k = s}^n {\frac{{f(k)}}{{2^{p(k)} }}\binom{{2p(k)}}{{p(k)}}}  = \sum_{k = \left\lfloor {(m + 2)/2} \right\rfloor }^{\left\lceil {r/2} \right\rceil } {\frac{{g(2k - 1)}}{{2^{q(2k - 1)} }}\binom{{q(2k - 1)}}{{q(2k - 1)/2}}},
\end{equation}

\item Suppose that, for every integer $j$, each of $p(2j)$ and $p(2j-1)$ is a sequence of non-negative integers having a definite parity but such that the parity of $p(2j)$ is different from the parity of $p(2j-1)$ for every integer $j$. 

If $p(2j)$ is an even integer for every integer $j$, then
\begin{equation}
\sum_{k = m}^r {\frac{{g(k)( - 1)^{q(k)} }}{{2^{q(k)} }}\binom{{2q(k)}}{{q(k)}}}  = \sum_{k = \left\lfloor {(s + 1)/2} \right\rfloor }^{\left\lfloor {n/2} \right\rfloor } {\frac{{ f(2k)}}{{2^{p(2k)} }}\binom{{p(2k)}}{{p(2k)/2}}} ,
\end{equation}
while if $p(2j)$ is an odd integer for every integer $j$, then
\begin{equation}
\sum_{k = m}^r {\frac{{g(k)( - 1)^{q(k)} }}{{2^{q(k)} }}\binom{{2q(k)}}{{q(k)}}}  = \sum_{k = \left\lfloor {(s + 2)/2} \right\rfloor }^{\left\lceil {n/2} \right\rceil } {\frac{{( - 1)^{p(2k - 1)} f(2k - 1)}}{{2^{p(2k - 1)} }}\binom{{p(2k - 1)}}{{p(2k - 1)/2}}} .
\end{equation}

\end{enumerate}

\begin{corollary}
Let an arbitrary polynomial identity have the following form:
\begin{equation}\label{eq.a1lk6eb}
P(t,\ldots)=\sum_{k = s}^n {f(k)\left( {1 + t} \right)^k }  = \sum_{k = m}^r {g(k)\,t^k },
\end{equation}
where $m$, $n$, $r$ and $s$ are non-negative integers, $f(k)$ and $g(k)$ are sequences, and $t$ is a complex variable. Let $v$ be an arbitrary real number. Then
\begin{equation}\label{eq.ly7bawk}
\sum_{k = s}^n {\frac{{f(k)}}{{2^k }}\binom{{2k + v}}{{\left( {2k + v} \right)/2}}\binom{{k + v}}{{v/2}}^{ - 1} }  = \sum_{k = \left\lfloor {(m + 1)/2} \right\rfloor }^{\left\lfloor {r/2} \right\rfloor } {\frac{{g(2k)}}{{2^{2k} }}\binom{{2k}}{k}\binom{{\left( {2k + v} \right)/2}}{{v/2}}^{-1}} ,
\end{equation}
and
\begin{equation}
\sum_{k = m}^r {\frac{{g(k)( - 1)^k }}{{2^k }}\binom{{2k + v}}{{\left( {2k + v} \right)/2}}\binom{{k + v}}{{v/2}}^{ - 1}}  = \sum_{k = \left\lfloor {(s + 1)/2} \right\rfloor }^{\left\lfloor {n/2} \right\rfloor } {\frac{{ f(2k)}}{{2^{2k} }}\binom{{2k}}{k}\binom{{\left( {2k + v} \right)/2}}{{v/2}}^{ - 1} } .
\end{equation}
\end{corollary}
In particular,
\begin{equation}
\sum_{k = s}^n {\frac{{f(k)}}{{2^k }}\binom{{2k}}{k}}  = \sum_{k = \left\lfloor {(m + 1)/2} \right\rfloor }^{\left\lfloor {r/2} \right\rfloor } {\frac{{g(2k)}}{{2^{2k} }}\binom{{2k}}{k}} ,
\end{equation}
and
\begin{equation}
\sum_{k = m}^r {\frac{{g(k)( - 1)^k }}{{2^k }}\binom{{2k}}{k}}  = \sum_{k = \left\lfloor {(s + 1)/2} \right\rfloor }^{\left\lfloor {n/2} \right\rfloor } {\frac{{ f(2k)}}{{2^{2k} }}\binom{{2k}}{k}} .
\end{equation}
\section{Polynomial identities}\label{polynomials}
In this section, by following the procedures outlined in Section~\ref{sec.combinat}, we derive new polynomial identities associated with the binomial theorem.

\begin{theorem}
Let $u$ and $v$ be arbitrary complex numbers such that $\Re u>-1$ and $\Re v>-1$. Let $x$ be a complex variable. If $n$ is a non-negative integer, then
\begin{equation}\label{eq.sf62i7f}
\begin{split}
&\sum_{k = 0}^n {( - 1)^{n - k} \binom{{n}}{k}\binom{{k + u + v + 1}}{{u + 1}}^{ - 1} \left( {1 - x} \right)^{n - k} }  \\&\qquad= \frac{{u + 1}}{{v + 1}}\sum_{k = 0}^n {( - 1)^k \binom{{n}}{k}\binom{{k + u + v + 1}}{{v + 1}}^{ - 1} x^{n - k}}.
\end{split}
\end{equation}

\end{theorem}
In particular,
\begin{equation}
\sum_{k = 0}^n {( - 1)^{n - k} \dfrac{{\binom{{n}}{k}}}{{k + 1}}\left( {1 - x} \right)^{n - k} }  = \sum_{k = 0}^n {( - 1)^k \dfrac{{\binom{{n}}{k}}}{{k + 1}}x^{n - k} } .
\end{equation}

\begin{proof}
Consider the following variation on the binomial theorem:
\begin{equation}\label{eq.ltr1okl}
\sum_{k = 0}^n {( - 1)^{n - k} \binom{{n}}{k}\left( {1 + t} \right)^k \left( {1 - x} \right)^{n - k} }  = \sum_{k = 0}^n {\binom{{n}}{k}t^k x^{n - k} }.
\end{equation}
Use~\eqref{eq.rxkjh8x} with 
\begin{equation*}
f(k) = ( - 1)^{n - k} \binom{n}{k}(1 - x)^{n - k}, \quad g(k) = \binom{n}{k}x^{n - k},\quad s=0=m,\quad r=n,
\end{equation*}
to obtain~\eqref{eq.sf62i7f}.

\end{proof}

\begin{theorem}
If $n$ is a non-negative integer, $v$ is a real number and $x$ is a complex variable, then
\begin{equation}\label{eq.qd43spp}
\begin{split}
&\sum_{k = 0}^n {( - 1)^{n - k} \binom{{n}}{k}2^{ - k} \binom{{2k + v}}{{\left( {2k + v} \right)/2}}\binom{{k + v}}{{v/2}}^{ - 1} \left( {1 - x} \right)^{n - k} }\\ 
&\qquad = \sum_{k = 0}^{\left\lfloor {n/2} \right\rfloor } {\binom{{n}}{{2k}}2^{ - 2k} \binom{{2k}}{k}\binom{{\left( {2k + v} \right)/2}}{k}^{ - 1} x^{n - 2k} } .
\end{split}
\end{equation}

\end{theorem}

\begin{proof}
On account of~\eqref{eq.a1lk6eb} and with~\eqref{eq.ltr1okl} in mind, use~\eqref{eq.ly7bawk} with 
\begin{equation*}
f(k) = ( - 1)^{n - k} \binom{n}{k}(1 - x)^{n - k}, \quad g(k) = \binom{n}{k}x^{n - k},\quad s=0=m,\quad r=n,
\end{equation*}
to obtain~\eqref{eq.qd43spp}.

\end{proof}

\begin{corollary}
If $n$ is a non-negative integer and $x$ is a complex variable, then
\begin{gather}
\sum_{k = 0}^n {( - 1)^{n - k} \binom{{n}}{k}\frac{{2^k }}{{k + 1}}\left( {1 - x} \right)^{n - k}  = } \sum_{k = 0}^{\left\lfloor {n/2} \right\rfloor } {\frac{{\binom{{n}}{{2k}}}}{{2k + 1}}x^{n - 2k} },\label{eq.hmx1w7h} \\
\sum_{k = 0}^n {( - 1)^{n - k} \frac{{\binom{{n}}{k}}}{{k + 2}}\,2^{ - k} \binom{{2\left( {k + 1} \right)}}{{k + 1}}\left( {1 - x} \right)^{n - k} }  = \sum_{k = 0}^{\left\lfloor {n/2} \right\rfloor } {\frac{{\binom{{n}}{{2k}}}}{{k + 1}}2^{ - 2k} \binom{{2k}}{k}x^{n - 2k} }\label{eq.bv1inky} .
\end{gather}
\end{corollary}

\begin{proof}
Identity~\eqref{eq.l5xib79} on page~\pageref{eq.l5xib79} and identities~\eqref{eq.hmx1w7h} and~\eqref{eq.bv1inky} correspond to the evaluation of~\eqref{eq.qd43spp} at $v=0$, $v=1$ and $v=2$, respectively.

In deriving~\eqref{eq.hmx1w7h}, we used~\eqref{eq.muz1im8}--\eqref{eq.mv0q30s} to obtain
\begin{equation*}
\binom{{2k + 1}}{{k + 1/2}} =\frac{{2^{4k + 4} }}{{\pi \left( {k + 1} \right)}}\binom{{2\left( {k + 1} \right)}}{{k + 1}}^{ - 2} \binom{{2k + 1}}{k},
\end{equation*}

\begin{equation*}
\binom{{k + 1}}{{1/2}} = \frac{{2^{2k + 3} }}{\pi }\binom{{2\left( {k + 1} \right)}}{{k + 1}}^{ - 1} ,
\end{equation*}
and
\begin{equation*}
\binom{{2k + 1}}{k} = \frac12\binom{{2\left( {k + 1} \right)}}{{k + 1}}.
\end{equation*}

\end{proof}

\begin{theorem}
Let $u$ and $v$ be complex numbers such that $\Re u>-1$ and $\Re v>-1$. If $n$ is a non-negative integer, then
\begin{equation}\label{eq.tnrm6l2}
\begin{split}
&\binom{{v}}{{v/2}}\sum_{k = 0}^n {( - 1)^{n - k} \binom{{n}}{k}2^{ - 2k} \binom{{2k + u}}{{\left( {2k + u} \right)/2}}\binom{{\left( {2k + u + v} \right)/2}}{{v/2}}^{-1}\left( {1 - x} \right)^{n - k} } \\
&\qquad = \binom{{u}}{{u/2}}\sum_{k = 0}^n {( - 1)^k \binom{{n}}{k}2^{ - 2k} \binom{{2k + v}}{{\left( {2k + v} \right)/2}}\binom{{\left( {2k + u + v} \right)/2}}{{u/2}}^{-1}x^{n - k} } .
\end{split}
\end{equation}

\end{theorem}

\begin{proof}
With~\eqref{eq.ltr1okl} in mind, use
\begin{equation*}
f(k) = ( - 1)^{n - k} \binom{n}{k}(1 - x)^{n - k}, \quad g(k) = \binom{n}{k}x^{n - k},\quad s=0=m,\quad r=n,
\end{equation*}
and $p(k)=k=q(k)$ in~\eqref{eq.cnwt5zb}.

\end{proof}

\begin{corollary}
If $n$ is a non-negative integer, then
\begin{gather}
\sum_{k = 0}^n {( - 1)^{n - k} \binom{{2k}}{k}2^{ - 2k} \binom{{n}}{k}\left( {1 - x} \right)^{n - k} }  = \sum_{k = 0}^n {( - 1)^k \binom{{2k}}{k}2^{ - 2k} \binom{{n}}{k}x^{n - k} }, \\
\sum_{k = 0}^n {( - 1)^{n - k} 2^{ - 2k} \binom{{n}}{k}C_{k + 1} \left( {1 - x} \right)^{n - k} }  = \sum_{k = 0}^n {( - 1)^k 2^{ - 2k} \binom{{n}}{k}C_{k + 1} x^{n - k} } .
\end{gather}
\end{corollary}

\begin{proof}
Evaluate~\eqref{eq.tnrm6l2} at $u=0=v$ and at $u=2=v$.
\end{proof}

\begin{theorem}
if $n$ is a non-negative integer, $v$ is a real number and $x$ is a complex variable, then
\begin{equation}\label{eq.p9vcynz}
\begin{split}
&\sum_{k = 0}^n {\binom{{n}}{k}2^{ - k} \binom{{2k + v}}{{\left( {2k + v} \right)/2}}\binom{{k + v}}{{v/2}}^{ - 1} \left( {1 - x} \right)^k x^{n - k} }\\ 
&\qquad = \sum_{k = 0}^{\left\lfloor {n/2} \right\rfloor } {\binom{{n}}{{2k}}2^{ - 2k} \binom{{2k}}{k}\binom{{\left( {2k + v} \right)/2}}{k}^{ - 1} (1 - x)^{2k} } .
\end{split}
\end{equation}

\end{theorem}

\begin{proof}
Consider another variation on the binomial theorem:
\begin{equation}\label{eq.n47svms}
\sum_{k = 0}^n {\binom{{n}}{k}\left( {1 - x} \right)^k \left( {1 + y} \right)^k x^{n - k} }=\sum_{k = 0}^n {\binom{{n}}{k}y^k \left( {1 - x} \right)^k }.
\end{equation}
This identity has the form of~\eqref{eq.a1lk6eb}. Use~\eqref{eq.ly7bawk} with 
\begin{equation*}
f(k) = \binom{n}{k}(1 - x)^k x^{n - k} ,\quad g(k) = \binom{n}{k}(1 - x)^k ,\quad s = 0 = m,\quad r = n,
\end{equation*}
to obtain~\eqref{eq.p9vcynz}.

\end{proof}

\begin{corollary}
If $n$ is a non-negative integer and $x$ is a complex variable, then
\begin{gather}
\sum_{k = 0}^n {\binom{{n}}{k}2^{ - k} \binom{{2k}}{k}\left( {1 - x} \right)^k x^{n - k} }  = \sum_{k = 0}^{\left\lfloor {n/2} \right\rfloor } {\binom{{n}}{{2k}}2^{ - 2k} \binom{{2k}}{k}(1 - x)^{2k} },\label{eq.yvskdge}\\ 
\sum_{k = 0}^n {\binom{{n}}{k}\frac{{2^k }}{{k + 1}}\left( {1 - x} \right)^k x^{n - k}  = } \sum_{k = 0}^{\left\lfloor {n/2} \right\rfloor } {\frac{{\binom{{n}}{{2k}}}}{{2k + 1}}(1 - x)^{2k} },\label{eq.ni2yglt} \\
\sum_{k = 0}^n {\binom{{n}}{k}2^{ - k} C_{k + 1} \left( {1 - x} \right)^k x^{n - k} }  = \sum_{k = 0}^{\left\lfloor {n/2} \right\rfloor } {\binom{{n}}{{2k}}2^{ - 2k} C_k \left( {1 - x} \right)^{2k} } \label{eq.fe8atkt} .
\end{gather}
\end{corollary}

\begin{proof}
Identities~\eqref{eq.yvskdge},~\eqref{eq.ni2yglt} and~\eqref{eq.fe8atkt} correspond to the evaluation of~\eqref{eq.p9vcynz} at $v=0$, $v=1$ and $v=2$, respectively.

\end{proof}

\begin{remark}
The reader is invited to employ the procedures established in Theorems~\ref{thm.a2bugv6}--\ref{thm.bgu7pnr} to discover more polynomial identities associated with~\eqref{eq.n47svms}.
\end{remark}

\section{More combinatorial identities}\label{combinatorial}

\subsection{Identities from the binomial theorem}

\begin{theorem}
Let $u$ and $v$ be complex numbers such that $\Re u>-1$ and $\Re v>-1$. if $n$ is a non-negative integer, then
\begin{equation}
\sum_{k = 0}^n {( - 1)^k \binom{{n}}{k}\binom{{k + u + v + 1}}{{u + 1}}^{ - 1} }  = \frac{{u + 1}}{{v + 1}}\,\binom{{n + u + v + 1}}{{v + 1}}^{ - 1}.
\end{equation}
\end{theorem}

\begin{proof}
Set $x=0$ in~\eqref{eq.sf62i7f}.
\end{proof}

\begin{theorem}
If $n$ is an integer and $v$ is a real number, then
\begin{equation}\label{eq.ilndvr6}
\begin{split}
&\sum_{k = 0}^n {( - 1)^{k} \binom{{n}}{k}2^{n - 2k} \binom{{2k + v}}{{\left( {2k + v} \right)/2}}\binom{{k + v}}{{v/2}}^{ - 1} }= \sum_{k = 0}^{\left\lfloor {n/2} \right\rfloor } {\binom{{n}}{{2k}}2^{ - 2k} \binom{{2k}}{k}\binom{{\left( {2k + v} \right)/2}}{k}^{ - 1} } ,
\end{split}
\end{equation}
and
\begin{equation}\label{eq.alpigwp}
\sum_{k = 0}^n {\binom{{n}}{k}2^{ - k} \binom{{2k + v}}{{\left( {2k + v} \right)/2}}\binom{{k + v}}{{v/2}}^{ - 1} }=\sum_{k = 0}^{\left\lfloor {n/2} \right\rfloor } {\binom{{n}}{{2k}}2^{n - 4k} \binom{{2k}}{k}\binom{{\left( {2k + v} \right)/2}}{k}^{ - 1}  } .
\end{equation}

\end{theorem}

\begin{proof}
Evaluate~\eqref{eq.qd43spp} at $x=-1$ and $x=2$, respectively. 
\end{proof}

\begin{remark}
Setting $x=0$ in~\eqref{eq.qd43spp} reproduces identity~\eqref{eq.hdj69wz}  while setting $x=1$ reproduces~\eqref{eq.y6pnymc}. 
\end{remark}

\begin{proposition}\label{prop.nkjtw6g}
If $n$ is a non-negative integer, then
\begin{gather}
\sum_{k = 0}^n {( - 1)^k \binom{{n}}{k}2^{n - 2k} \binom{{2k}}{k}}  = \sum_{k = 0}^{\left\lfloor {n/2} \right\rfloor } {\binom{{n}}{{2k}}2^{ - 2k} \binom{{2k}}{k}} ,\\
\sum_{k = 0}^{\left\lfloor {n/2} \right\rfloor } {\binom{{n}}{{2k}}\frac{1}{{2k + 1}}}  = \frac{{2^{n - 1} }}{{2^n  - 1}}\sum_{k = 1}^{\left\lceil {n/2} \right\rceil } {\binom{{n}}{{2k - 1}}\frac{1}{k}},\quad n\ne 0,\\ 
\sum_{k = 0}^n {( - 1)^k \binom{{n}}{k}\frac{{2k + 1}}{{k + 2}}2^{n - 2k + 1} C_k }  = \sum_{k = 0}^{\left\lfloor {n/2} \right\rfloor } {\binom{{n}}{{2k}}2^{ - 2k} C_k } .
\end{gather}

\end{proposition}

\begin{proof}
Set $x=-1$ in identities~\eqref{eq.l5xib79},~\eqref{eq.hmx1w7h} and~\eqref{eq.bv1inky}.
\end{proof}

\begin{proposition}
If $n$ is a non-negative integer, then
\begin{gather}
\sum_{k = 0}^{\left\lfloor {n/2} \right\rfloor } {\binom{{n}}{{2k}}2^{ - 2k} \binom{{2k}}{k} }=2^{-n}\binom{2n}n,\nonumber\\
\sum_{k = 0}^{\left\lfloor {n/2} \right\rfloor } {\frac{{\binom{{n}}{{2k}}}}{{2k + 1}}}  = \frac{{2^n }}{{n + 1}},\\
\sum_{k = 0}^{\left\lfloor {n/2} \right\rfloor } {\frac{{\binom{{n}}{{2k}}}}{{k + 1}}2^{ - 2k} \binom{{2k}}{k}}  = \frac{{2^{ - n + 1} }}{{n + 2}}\left( {2n + 1} \right)C_n .
\end{gather}

\end{proposition}

\begin{proof}
Set $x=1$ in identities~\eqref{eq.l5xib79},~\eqref{eq.hmx1w7h} and~\eqref{eq.bv1inky}.
\end{proof}

\begin{proposition}\label{prop.n12zl7g}
If $n$ is a non-negative integer, then
\begin{gather}
\sum_{k = 0}^n {\binom{{n}}{k}2^{ - k} \binom{{2k}}{k}}  = \sum_{k = 0}^{\left\lfloor {n/2} \right\rfloor } {\binom{{n}}{{2k}}2^{n - 4k} \binom{{2k}}{k}},\\
\sum_{k = 0}^{\left\lfloor {n/2} \right\rfloor } {\binom{{n}}{{2k}}\frac{{2^{n - 2k + 1}  - 2^{2k + 1} }}{{2k + 1}}}  = \sum_{k = 1}^{\left\lceil {n/2} \right\rceil } {\binom{{n}}{{2k - 1}}\frac{{2^{2k - 1} }}{k}},\\
\sum_{k = 0}^n {\binom{{n}}{k}\frac{{2k + 1}}{{k + 2}}\,2^{-k + 1} C_k }  = \sum_{k = 0}^{\left\lfloor {n/2} \right\rfloor } {\binom{{n}}{{2k}}2^{n - 4k} C_k } .
\end{gather}
\end{proposition}

\begin{proof}
Set $x=2$ in identities~\eqref{eq.l5xib79},~\eqref{eq.hmx1w7h} and~\eqref{eq.bv1inky}.

\end{proof}

\begin{theorem}
if $n$ is a non-negative integer and $v$ is a real number, then
\begin{equation}\label{eq.yveoyay}
\sum_{k = 0}^n {(-1)^{n - k}\binom{{n}}{k}\binom{{2k + v}}{{\left( {2k + v} \right)/2}}\binom{{k + v}}{{v/2}}^{ - 1} }=\sum_{k = 0}^{\left\lfloor {n/2} \right\rfloor } {\binom{{n}}{{2k}}\binom{{2k}}{k}\binom{{\left( {2k + v} \right)/2}}{k}^{ - 1} } .
\end{equation}

\end{theorem}

\begin{proof}
Set $x=-1$ in~\eqref{eq.p9vcynz}.
\end{proof}

\begin{proposition}\label{prop.vq5bsch}
If $n$ is a non-negative integer, then
\begin{gather}
\sum_{k = 0}^n {( - 1)^{n - k} \binom{{n}}{k}\binom{{2k}}{k}}  = \sum_{k = 0}^{\left\lfloor {n/2} \right\rfloor } {\binom{{n}}{k}\binom{{n - k}}{k}}, \\
\sum_{k = 0}^n {( - 1)^{n - k} \binom{{n}}{k}\frac{{2^{2k} }}{{k + 1}}}  = \sum_{k = 0}^{\left\lfloor {n/2} \right\rfloor } {\binom{{n}}{{2k}}\frac{{2^{2k} }}{{2k + 1}}} ,\\
\sum_{k = 0}^n {( - 1)^{n - k} \binom{{n}}{k}\frac{{2\left( {2k + 1} \right)}}{{k + 2}}C_k }  = \sum_{k = 0}^{\left\lfloor {n/2} \right\rfloor } {\binom{{n}}{{2k}}C_k } .
\end{gather}
\end{proposition}

\begin{proof}
Set $x=-1$ in each of identities~\eqref{eq.yvskdge}--~\eqref{eq.fe8atkt} or what is the same thing, $v=0$, $v=1$ and $v=2$ in~\eqref{eq.yveoyay}.

\end{proof}

\begin{theorem}
If $n$ is a non-negative integer and $v$ is a real number, then
\begin{equation}\label{eq.ei75ly5}
\begin{split}
&\sum_{k = 0}^n {\binom{{2n}}{{2k}}\binom{{2\left( {n - k} \right)}}{{n - k}}\binom{{2k + v}}{{\left( {2k + v} \right)/2}}\binom{{\left( {2n + v} \right)/2}}{{n - k}}} ^{ - 1}\\&\qquad  = \sum_{k = 0}^{\left\lfloor {n/2} \right\rfloor } {\binom{{n}}{{2k}}2^{2n - 2k} \binom{{2k}}{k}\binom{{2k + v}}{{\left( {2k + v} \right)/2}}\binom{{\left( {4k + v} \right)/2}}{k}^{ - 1} } 
\end{split}
\end{equation}
and
\begin{equation}\label{eq.lwx1yzw}
\begin{split}
&\sum_{k = 1}^n {\binom{{2n}}{{2k - 1}}\binom{{2\left( {n - k + 1} \right)}}{{n - k + 1}}\binom{{2k - 1 + v}}{{\left( {2k - 1 + v} \right)/2}}\binom{{\left( {2n + v + 1} \right)/2}}{{n - k + 1}}} ^{ - 1}\\ &\qquad  =\sum_{k = 1}^{\left\lceil {n/2} \right\rceil } {\binom{{n}}{{2k - 1}}2^{2n + 1 - 2k} \binom{{2k}}{k}\binom{{2k - 1 + v}}{{\left( {2k - 1 + v} \right)/2}}\binom{{\left( {4k + v - 1} \right)/2}}{k}^{ - 1} }.
\end{split}
\end{equation}

\end{theorem}
In particular,
\begin{equation}
\sum_{k = 0}^n {\binom{{2n}}{{2k}}\binom{{2\left( {n - k} \right)}}{{n - k}}\binom{{2k}}{k}\binom{{n}}{k}} ^{ - 1}  = \sum_{k = 0}^{\left\lfloor {n/2} \right\rfloor } {\binom{{n}}{{2k}}2^{2n - 2k} \binom{{2k}}{k}}
\end{equation}
and
\begin{equation}
\begin{split}
&\sum_{k = 1}^n {\binom{{2n}}{{2k - 1}}\binom{{2\left( {n - k + 1} \right)}}{{n - k + 1}}\binom{{2k - 1}}{{\left( {2k - 1} \right)/2}}\binom{{\left( {2n + 1} \right)/2}}{{n - k + 1}}} ^{ - 1}\\&\qquad  = \sum_{k = 1}^{\left\lceil {n/2} \right\rceil } {\binom{{n}}{{2k - 1}}2^{2n + 1 - 2k} \binom{{2k}}{k}\binom{{2k - 1}}{{\left( {2k - 1} \right)/2}}\binom{{\left( {4k - 1} \right)/2}}{k}^{ - 1} } .
\end{split}
\end{equation}

\begin{proof}
Since
\begin{equation*}
\left( {\cos \left( {\frac{x}{2}} \right) + \sin \left( {\frac{x}{2}} \right)} \right)^{2m}=\left( {1 + \sin x} \right)^m,
\end{equation*}
the binomial theorem gives
\begin{equation*}
\sum_{k = 0}^{2n} {\binom{{2n}}{k}\cos ^{2n - k} \left( {\frac{x}{2}} \right)\sin ^k \left( {\frac{x}{2}} \right)}  = \sum_{k = 0}^n {\binom{{n}}{k}\sin ^k x},
\end{equation*}
so that
\begin{equation*}
\sum_{k = 0}^{2n} {\binom{{2n}}{k}\cos ^{2n - k} x\sin ^k x}  = \sum_{k = 0}^n {\binom{{n}}{k}2^k \cos ^k x\sin ^k x}
\end{equation*}
and, therefore,
\begin{equation*}
\begin{split}
&\sum_{k = 0}^n {\binom{{2n}}{{2k}}\sin ^{2k + v} x\cos ^{2n - 2k} x}  + \sum_{k = 1}^n {\binom{{2n}}{{2k - 1}}\sin ^{2k - 1 + v} x\cos ^{2n - 2k + 1} x} \\
&\qquad = \sum_{k = 0}^{\left\lfloor {n/2} \right\rfloor } {\binom{{n}}{{2k}}2^{2k} \cos ^{2k} x\sin ^{2k + v} x}  + \sum_{k = 1}^{\left\lceil {n/2} \right\rceil } {\binom{{n}}{{2k - 1}}2^{2k - 1} \cos ^{2k - 1} x\sin ^{2k - 1 + v} x}.
\end{split}
\end{equation*}
Integrating from $0$ to $\pi$ using Lemma~\ref{lem.ibicayr} gives~\eqref{eq.ei75ly5} while multiplying through by $\cos x$ and integrating from $0$ to $\pi$ gives~\eqref{eq.lwx1yzw}.

\end{proof}

\subsection{Identities from Waring's formulas}

Waring's formula and its dual \cite[Equations (22) and (1)]{gould99} are
\begin{equation}\label{eq.h35j76y}
\sum_{k = 0}^{\left\lfloor {n/2} \right\rfloor } {( - 1)^k \frac{n}{{n - k}}\binom {n-k}k(xy)^k (x + y)^{n - 2k} }  = x^{n}  + y^{n}
\end{equation}
and
\begin{equation}\label{eq.amsa61r}
\sum_{k = 0}^{\left\lfloor {n/2} \right\rfloor } {( - 1)^k \binom {n-k}k(xy)^k (x + y)^{n - 2k} }  = \frac{x^{n + 1}  - y^{n + 1}}{x - y}\,.
\end{equation}
Identity \eqref{eq.h35j76y} holds for positive integer $n$ while identity \eqref{eq.amsa61r} holds for any non-negative integer $n$.

\begin{theorem}
If $n$ is a non-negative integer and $v$ is a real number, then
\begin{equation}\label{eq.rvlh5im}
\begin{split}
&\sum_{k = 0}^{\left\lfloor {n/2} \right\rfloor } {( - 1)^k \frac{n}{{n - k}}\binom{{n - k}}{k}2^{ - 4k} \binom{{2k + v}}{{\left( {2k + v} \right)/2}}}\\ &\qquad = \binom{{2n + v}}{{\left( {2n + v} \right)/2}}\binom{{v}}{{v/2}}2^{1 - 2n} \binom{{n + v}}{{v/2}}^{ - 1} .
\end{split}
\end{equation}
\end{theorem}
In particular,
\begin{equation}
\sum_{k = 0}^{\left\lfloor {n/2} \right\rfloor } {( - 1)^k \frac{n}{{n - k}}\binom{{n - k}}{k}2^{ - 4k} \binom{{2k}}{k}}  = 2^{ - 2n + 1} \binom{{2n}}{n}.
\end{equation}

\begin{proof}
Write $\cos^2(x/2)$ for $x$ and $\sin^2(y/2)$ for $y$ in~\eqref{eq.h35j76y} and multiply through by $\sin^vx$ to obtain
\begin{equation*}
\begin{split}
\sum_{k = 0}^{\left\lfloor {n/2} \right\rfloor } {( - 1)^k \frac{n}{{n - k}}\binom{{n - k}}{k}2^{ - 2k} \sin ^{2k + v} x}  &= 2^v \cos ^{2n + v} \left( {\frac{x}{2}} \right)\sin ^v \left( {\frac{x}{2}} \right)\\
&\qquad + 2^v \sin ^{2n + v} \left( {\frac{x}{2}} \right)\cos ^v \left( {\frac{x}{2}} \right),
\end{split}
\end{equation*}
from which upon term-wise integration from $0$ to $\pi$, identity~\eqref{eq.rvlh5im} follows.

\end{proof}

By writing $\cos^2(x/2)$ for $x$ and $-\sin^2(y/2)$ for $y$, the reader is invited to discover a combinatorial identity associated with~\eqref{eq.amsa61r}.

\subsection{Identities from an identity of Simons}

Simons~\cite{simons01} proved an identity that is equivalent to the following:
\begin{equation}\label{eq.oli5mgr}
\sum_{k = 0}^n {( - 1)^{n - k} \binom{{n}}{k}\binom{{n + k}}{k}\left( {1 + t} \right)^k }  = \sum_{k = 0}^n {\binom{{n}}{k}\binom{{n + k}}{k}t^k } .
\end{equation}
On choosing 
\begin{equation}
f(k) = ( - 1)^{n - k} \binom{{n}}{k}\binom{{n + k}}{k},\quad g(k) = \binom{{n}}{k}\binom{{n + k}}{k},
\end{equation}
$s=m=0$ and $r=n$ in~\eqref{eq.a1lk6eb},~\eqref{eq.ly7bawk} gives the result stated in the next proposition.
\begin{proposition}
If $n$ is a non-negative integer and $v$ is a real number, then
\begin{equation}\label{eq.s9xhgba}
\begin{split}
&\sum_{k = 0}^n {( - 1)^{n - k} \binom{{n}}{k}2^{ - k} \binom{{n + k}}{k}\binom{{2k + v}}{{\left( {2k + v} \right)/2}}\binom{{k + v}}{{v/2}}^{ - 1} }\\
&\qquad\qquad  = \sum_{k = 0}^{\left\lfloor {n/2} \right\rfloor } {\binom{{n}}{{2k}}2^{ - 2k} \binom{{n + 2k}}{{2k}}\binom{{2k}}{k}\binom{{\left( {2k + v} \right)/2}}{{v/2}}^{-1}} .
\end{split}
\end{equation}

\end{proposition}

In particular,
\begin{equation}
\sum_{k = 0}^n {( - 1)^{n - k} \binom{{n}}{k}2^{ - k} \binom{{n + k}}{k}\binom{{2k}}{k}}  = \sum_{k = 0}^{\left\lfloor {n/2} \right\rfloor } {\binom{{n}}{{2k}}2^{ - 2k} \binom{{n + 2k}}{{2k}}\binom{{2k}}{k}} .
\end{equation}

The same set of sequences and parameters, $f(k)$ etc.~that led to~\eqref{eq.s9xhgba}, when used in~\eqref{eq.cnwt5zb} gives the following result.
\begin{proposition}
If $n$ is a non-negative integer and $u$ and $v$ are real numbers, then
\begin{equation}
\begin{split}
&\sum_{k = 0}^n {( - 1)^{n - k} \binom{{n}}{k}2^{ - 2k} \binom{{n + k}}{k}\binom{{v}}{{v/2}}\binom{{2k + u}}{{\left( {2k + u} \right)/2}}\binom{{\left( {2k + u + v} \right)/2}}{{v/2}}^{ - 1} } \\
&\qquad = \sum_{k = 0}^n {( - 1)^k \binom{{n}}{k}2^{ - 2k} \binom{{n + k}}{k}\binom{{u}}{{u/2}}\binom{{2k + v}}{{\left( {2k + v} \right)/2}}\binom{{\left( {2k + u + v} \right)/2}}{{u/2}}^{ - 1} } .
\end{split}
\end{equation}
\end{proposition}

\bigskip
\hrule
\bigskip





\end{document}